\documentclass[hidelinks]{amsart}
\usepackage[utf8]{inputenc}

\usepackage{amssymb}
\usepackage{mathrsfs}
\usepackage{amsmath}
\usepackage{amsthm}
\usepackage{tikz-cd}
\usepackage{bbm}
\usepackage{mathdots}
\usepackage[shortlabels]{enumitem}
\usepackage[normalem]{ulem}
\usepackage[normalem]{ulem}
\usepackage[backend=biber]{biblatex}
\usepackage{hyperref} 

\bibliography{zoterobib}

\usetikzlibrary{babel,decorations.pathreplacing}

\allowdisplaybreaks

\begingroup
    \makeatletter
    \@for\theoremstyle:=definition,remark,plain\do{%
        \expandafter\g@addto@macro\csname th@\theoremstyle\endcsname{%
            \addtolength\thm@preskip\parskip
            }%
        }
\endgroup			 			

\newcounter{theorem}\numberwithin{theorem}{section}

\newcounter{proofcount} 
\newtheorem{claim}{Claim}
\newtheorem*{nclaim}{Claim}
\makeatletter
\@addtoreset{claim}{proofcount}
\makeatother

\AtBeginEnvironment{proof}{\stepcounter{proofcount}}
\theoremstyle{remark}

\makeatletter
\newtheorem*{cproof/}{Proof of claim \rev@cproofmark}

\newenvironment{cproof}[1][\@nil]
  {\def\@tmp{#1}%
   \ifx\@tmp\@nnil
       \def\rev@cproofmark{\theclaim}%
    \else
       \let\rev@cproofmark\@tmp
    \fi
   \pushQED{\qed}\begin{cproof/}}
  {\popQED\end{cproof/}}
\makeatother

\newtheorem*{ncproof/}{Proof of claim}
\newenvironment{ncproof}
  {%
   \pushQED{\qed}\begin{ncproof/}}
  {\popQED\end{ncproof/}}

\theoremstyle{plain}
\newtheorem{thm}[theorem]{Theorem}
\newtheorem{cor}[theorem]{Corollary}
\newtheorem{lem}[theorem]{Lemma}

\newtheorem*{thm*}{Theorem}

\theoremstyle{definition}
\newtheorem{defi/}[theorem]{Definition}
\newtheorem{obs/}[theorem]{Remark}
\newtheorem{ej/}[theorem]{Example}

\newenvironment{defi}
  {%
   \pushQED{\qed}\begin{defi/}}
  {\popQED\end{defi/}}

\counterwithout{equation}{section}

\newcommand{\thistheoremname}{}
\newtheorem{genericthm}[theorem]{\thistheoremname}

\newcommand{\restr}{\mathord{\upharpoonright}}

\newcommand{\seq}[1]{{\langle{#1}\rangle}}
\newcommand{\darrow}{\mathord{\downarrow}}
\newcommand{\uarrow}{\mathord{\uparrow}}

\renewcommand{\subset}{\subseteq}
\renewcommand{\supset}{\supseteq}
\renewcommand{\phi}{\varphi}
\renewcommand{\diamond}{\diamondsuit}

\newcommand*{\myarrow}{\mathrel{\rightarrowtail\kern-1.9ex\twoheadrightarrow}}

\newcommand\R{\mathbb{R}}

\newcommand\T{\tau}

\newcommand\Q{\mathbb{Q}}
\newcommand\dom{\mathrm{dom}}
\newcommand\ran{\mathrm{ran}}

\newcommand\w{\omega}
\newcommand\cf{\mathrm{cf}}
\newcommand\U{\mathcal{U}}
\newcommand\B{\mathcal{B}}
\newcommand\h{\mathrm{ht}}
\newcommand\D{\mathbb{D}}

\newcommand\I{\mathcal{I}}
\newcommand\J{\mathcal{J}}
\newcommand\V{\mathcal{V}}

\title{Square compactness and Lindelöf trees}
\author{Pedro E. Marun}
\address{DEPARTMENT OF MATHEMATICAL SCIENCES,
CARNEGIE MELLON UNIVERSITY,
PITTSBURGH, PA, 15213}
\email{pmarun@andrew.cmu.edu}
\subjclass[2020]{Primary 03E05;
	Secondary 54B10,03E04}
\date{}
\thanks{The results of this paper will form a part of the author’s PhD thesis written under the supervision of James Cummings, to whom the author would like to express his gratitude.}
 
\begin{document}

\begin{abstract}
We prove that every weakly square compact cardinal is a strong limit cardinal. We also study Aronszajn trees with no uncountable finitely splitting subtrees, characterizing them in terms of being Lindelöf with respect to a particular topology. We prove that the class of such trees is non-empty and lies strictly between the classes of Suslin and Aronszajn trees.
\end{abstract}

\maketitle

\section{Introduction}

Recall that a topological space is \emph{Lindelöf} if and only if every open cover has a countable subcover. Unlike compactness, the Lindelöf property need not be presserved by finite products, as shown by the classical Sorgenfrey line example, which is the space $X$ with underlying set $\R$ and topology generated by all left-closed right-open intervals. This space is Lindelöf, but the uncountable set $\{(x,-x):x\in\R\}$ is closed and discrete in $X^2$, hence $X^2$ is not Lindelöf. For details, see \cite[Countereexample 84]{seebachCounterexamplesTopology1978}.

Extending this to larger cardinals $\kappa$, we say that a topological space $X$ is \emph{$\kappa$-compact} if and only if every open cover of $X$ has a subcover of size less than $\kappa$. So, compact is $\aleph_0$-compact and Lindelöf is $\aleph_1$-compact. In connection with this, Hajnal and Juhász \cite{hajnalSquarecompactCardinals1973} introduced the following large cardinal notion: an infinite cardinal $\kappa$ is \emph{square compact} if and only if for every $\kappa$-compact space $X$, $X^2$ is $\kappa$-compact. This is in fact equivalent to the product of any two $\kappa$-compact spaces being $\kappa$-compact. For the non-trivial direction, if $X$ and $Y$ are $\kappa$-compact, then so is their disjoint sum $X\oplus Y$. By assumption, $(X\oplus Y)^2$ is $\kappa$-compact. Since $X\times Y$ is a closed subset of $(X\oplus Y)^2$, it follows that $X\times Y$ is $\kappa$-compact too.

Recall that the \emph{weight} $w(X)$ of a topological space $X$ is the least size of a base for the topology on $X$. A refined version of square compactness, graduated by weights, was introduced by Buhagiar and D\u{z}amonja in their 2021 paper \cite{buhagiarSquareCompactnessFilter2021}: given some cardinal $\lambda$, an infinite cardinal $\kappa$ is \emph{$\lambda$-square compact} if and only if for every space $X$ of weight $\le\lambda$, if $X$ is $\kappa$-compact, then $X^2$ is $\kappa$-compact
They (and we) say that $\kappa$ is \emph{weakly square compact} if and only if it is $\kappa$-square compact.

In this terminology, the results in \cite{hajnalSquarecompactCardinals1973} can be stated as:

\begin{thm*}[Hajnal-Juhász {\cite[Theorem 1]{hajnalSquarecompactCardinals1973}}]
Every weakly square compact cardinal is regular.
\end{thm*}

\begin{thm*}[Hajnal-Juhász {\cite[Theorem 2]{hajnalSquarecompactCardinals1973}}]
Suppose that $\kappa$ is uncountable. If $\kappa$ is $2^\kappa$-square compact, then it is weakly compact.
\end{thm*}

This shows that the existence of $\kappa$ which is $2^\kappa$-square compact is already a large cardinal notion. In \cite{buhagiarSquareCompactnessFilter2021}, Buhagiar and D\u{z}amonja undertake a closer study of weak square compactness, and give a variety of equivalent formulations. In particular, they proved the following:

\begin{thm*}[Buhagiar-D\u{z}amonja, {\cite[Theorem 5.1]{buhagiarSquareCompactnessFilter2021}}]
Let $\kappa$ be an uncountable cardinal. Suppose that $\kappa^{<\kappa}=\kappa$. Then $\kappa$ is weakly compact if and only if it is weakly square compact.
\end{thm*}

Note that if $\kappa$ is weakly compact, then $\kappa^{<\kappa}=\kappa$. The same is not obviously true if instead $\kappa$ is $\kappa$-square compact. The first result we establish in this paper is that indeed $\kappa=\kappa^{<\kappa}$ whenever $\kappa$ is $\kappa$-square compact, thereby removing the cardinal arithmetic assumption from the Buhagiar-D\u{z}amonja theorem. This is done by generalizing the Sorgenfrey line construction.

As far as we know, strong compactness continues to be the best upper bound for the consistency strength of full square compactness. This is part of the folklore of the subject, and for completeness we include a proof (see Theorem \ref{strong}). In their \cite[Theorem 2.10]{buhagiarSquareCompactnessFilter2021}, Buhagiar and Dzamonja give the following characterization: a cardinal $\kappa$ is strongly compact if and only if every $\kappa$-box product of $\kappa$-compact spaces is $\kappa$-compact. Recall that the \emph{$\kappa$-box product topology} on $\prod_i X_i$ is generated by all sets of the form $\prod_i U_i$ with $U_i$ open in $X_i$ and $|\{i: U_i\neq X_i\}|<\kappa$. Of course, \cite[Theorem 2.10]{buhagiarSquareCompactnessFilter2021} immediately yields the aforementioned folklore result.

Sharper upper bounds for the consistency strength of ``There exists $\kappa$ which is $2^\kappa$-square compact" appear in \cite[Theorem 6.1]{buhagiarSquareCompactnessFilter2021}. 

Having studied topologies on linear orders, we turn to looking at topologies on trees, with a view towards introducing new examples of Lindelöf spaces. In the survey \cite{nyikosVariousTopologiesTrees1997}, Nyikos considers a total of ten different topologies on trees. Of these, only two are always Hausdorff, and we adhere to the doctrine of only considering Hausdorff spaces. By \cite[Theorem 3.6]{nyikosVariousTopologiesTrees1997}, the \emph{coarse wedge topology} appears uninteresting for our purposes, since it is $\w_1$-compact\footnote{We caution the reader that, in Nyikos' survey, $X$ being $\w_1$-compact means that every closed and discrete subset of $X$ is countable.} if and only if the underlying tree has countably many minimal elements. This leaves us with only the \emph{fine wedge topology} to focus on. We give a tree-theoretic characterization of being Lindelöf with respect to this topology. First, some terminology: we say that a tree is \emph{finitely splitting} if and only if every point in the tree has finitely many immediate successors. A \emph{subtree} of a tree $T$ is a set $S\subset T$ such that for all $x\in S$ and $y\in T$, if $y<x$, then $y\in S$.

\begin{thm*}
Let $T$ be an $\aleph_1$-tree. Then $T$ is Lindelöf with respect to the fine-wedge topology if and only if every finitely splitting subtree of $T$ is countable. 
\end{thm*}

We shall show that
\[
\{\text{Suslin}\}\subset \{\text{Lindelöf}\} \subsetneq \{\text{Aronszajn}\}.
\]
Here, by Lindelöf we mean Lindelöf with respect to the fine-wedge topology.

Given a partially ordered set $X$, we let $X^*$ denote the \emph{dual order} on $X$, that is $x<^* y$ if and only if $y<x$.

A \textit{tree} is a pair $(T,<_T)$ such that $<_T$ is a strict partial order on $T$ and $\{y\in T:y<_Tx\}$ is well-ordered for every $x\in T$. We will usually suppress the subscript in $<_T$ and identify the tree with its underlying set when there is no danger of confusion.

Elements of a tree are referred to as \emph{nodes} or \emph{points}. We say $x,y\in T$ are \textit{comparable}, denoted $x\parallel y$, if and only if $x\le_T y$ or $y\le_T x$. Otherwise, we say that $x$ and $y$ are \textit{incomparable}, denoted $x\perp y$. The \textit{height} of a node $x\in T$ is the order-type of the set $\{y\in T:y<_T x\}$, denoted $\h_T(x)$ (or simply $\h(x)$). Given an ordinal $\alpha$, \textit{level} $\alpha$ of the tree is the set $T_\alpha=\{x\in T:\h_T(x)=\alpha\}$. The \textit{height} $h_T(T)$ of $T$ is defined by $\h(T)=\min\{\alpha:T_\alpha=\emptyset\}$. Given an ordinal $\alpha<\h(T)$, we let $T\restr\alpha=\{x\in T:\h(x)<\alpha\}$, which is of course a subtree of $T$ of height $\alpha$. Given $x\in T$, we let $I_T(x)$ denote the set of \emph{immediate successors} of $x$, and write $I(x)$ when there is no possibility of confusion. 

\section{Square compactness}

We shall say a space $X$ is \textit{hereditarily $\kappa$-compact} if and only if every subspace of $X$ is $\kappa$-compact. For example, any space of weight less than $\kappa$ is hereditarily $\kappa$-compact.

A useful criterion for hereditary $\kappa$-compactness is

\begin{lem}\label{2}
Let $(X,\tau)$ be a topological space. Then $X$ is hereditarily $\kappa$-compact if and only if for every $\mathcal U\subset\tau$ there is some $\U_0\in[\U]^{<\kappa}$ with $\bigcup\U=\bigcup\U_0$.
\end{lem}

\begin{proof}
$\Rightarrow$) Given $\U$, consider the subspace $\bigcup\U$.

$\Leftarrow$) Suppose $Y\subset X$ is not $\kappa$-compact. Fix some $\mathcal U\subset\tau$ such that $\bigcup\U\supset Y$ but there is no $\U_0\in[\U]^{<\kappa}$ with $\bigcup\U_0\supset Y$. Then $\bigcup\U_0\neq\bigcup \U$ for every $\U_0\in[\U]^{<\kappa}$.
\end{proof}

The following is obvious:

\begin{lem}\label{3}
Suppose that $(X,\tau)$ is $\kappa$-compact and $Y\subset X$ is closed. Then $Y$ is $\kappa$-compact with the subspace topology.
\end{lem}

As mentioned in the introduction, Hajnal and Juhász already proved that weak square compactness entails regularity. To deal with the (strong) inaccessibility of $\kappa$, we will generalize the classical construction of the Sorgenfrey line to larger linear orders.

\begin{defi}\label{5}
Let $(X,<)$ be a \emph{dlo} (dense linear order without end-points). The \textit{density} of $X$, denoted $d(X)$, is the cardinal
\[d(X)=\min\{|D|:D \text{ is dense in } X\}\]
This of course coincides with the density of $X$ as a topological spacer under the order topology.
\end{defi}
For example, $d(\R)=\aleph_0$. It is straightforward to show that $w(X)$, the weight of $X$ with respect to the order topology, is exactly $d(X)$.

\begin{defi}\label{8}
Given a dlo $(X,<)$, the family $\{[x,y):x,y\in X\wedge x<y\}$ forms a base for a topology on $X$, which we shall call the \textit{Sorgenfrey} topology.
\end{defi}

\begin{lem}\label{9}
Let $(X,<)$ be a dlo with $d(X)<\kappa$. Then the Sorgenfrey topology on $X$ is hereditarily $\kappa$-compact.
\end{lem}

\begin{proof}
Let $\U\subset \{[x,y):x,y\in X\}$ and let $W=\bigcup\{(x,y):[x,y)\in\U\}$. Obviously, $W$ is open with respect to the order topology on $X$, which has weight less than $\kappa$. By Lemma \ref{2}, $W=\bigcup\{(x,y):[x,y)\in\U_0\}$ for some $\U_0\in[\U]^{<\kappa}$. Let $A:=(\bigcup\U)\setminus W$.

\begin{nclaim}
$|A|<\kappa$.
\end{nclaim}

\begin{ncproof}
Fix $D\in[X]^{<\kappa}$ dense in the order topology. For each $x\in A$, find $[a_x,b_x)\in\U$ such that $x\in [a_x,b_x)$. Since $x\not\in W$, we infer that $x=a_x<b_x$, so we can pick some $d_x\in D$ with $x<d_x<b_x$. Now suppose $x,y\in A$ with $x<y$. Since $y\not\in W$, $b_x\le y$. Then $d_x<b_x\le y<d_y$, so $d_x<d_y$. Therefore, $x\mapsto d_x$ is an injective map from $A$ into $D$.
\end{ncproof}
For each $x\in A$, pick $U_x\in\U$ with $x\in U_x$. Let $\U_1=\{U_x:x\in A\}$. Clearly, $|\U_1|<\kappa$. We now have that $\U_2=\U_0\cup\U_1\in[\U]^{<\kappa}$ and $\bigcup\U_2=\bigcup\U$.\end{proof}

\begin{lem}\label{10}
Let $\kappa>\omega$ be a cardinal. Suppose that there is a dlo $(X,<)$ with $d(X)<\kappa=|X|$. Then $\kappa$ is not $\kappa$-square compact.
\end{lem}

\begin{proof}
Replacing $X$ by $X\oplus X^*$ if necessary, we may assume that $(X,<)$ admits an order reversing involution, which we shall suggestively denote by $x\mapsto -x$.

Let $\tau$ be the Sorgenfrey topology on $X$. Note that $w(X,\tau)\le\kappa$, because $|X|\le\kappa$, and that $(X,\tau)$ is $\kappa$-compact by lemma \ref{9}. It therefore suffices to show that $X^2$ is not $\kappa$-compact with respect to the product topology. Let
\[
Y=\{(x,-x):x\in X\}.
\]
Since $x\mapsto -x$ is order-reversing, it is continuous with respect to the order topology $\tau_<$, hence $Y$ is closed in $(X^2,\tau_<\otimes\tau_<)$.  But $\tau_<\subset\tau$, so $Y$ is closed in $(X^2,\tau\otimes\tau)$. For each $x\in X$, pick $u_x,v_x\in X$ with $u_x<x<v_x$. Now observe that
\[
([x,v_x)\times [-x,-u_x))\cap Y=\{(x,-x)\}.
\]

We have shown that $Y$ is discrete in $(X^2,\tau\otimes\tau)$. Since $|Y|=\kappa$, $Y$ is not $\kappa$-compact, and so neither is $(X^2,\tau\otimes\tau)$ because $Y$ is closed.
\end{proof}

The goal now is to build large dlo's with small density. This will be possible, under certain cardinal arithmetic constraints. Our original construction was rather convoluted, and we thank Will Brian for suggesting the following simpler approach.

\begin{lem}\label{12}
Let $\kappa\ge\omega_1$. Suppose there exist infinite cardinals $\mu$ and $\theta$ such that $\mu^{<\theta}=\mu<\kappa\le\mu^\theta$. Then there is a dlo $X$ with $d(X)<\kappa=|X|$.
\end{lem}

\begin{proof}
Let $Y:={}^\theta\mu$, ordered lexicographically. Note that $|Y|=\mu^\theta\ge\kappa$. Let $D$ be the set of sequences in $Y$ which are eventually $0$. Then $D$ is dense in $Y$ and $|D|=\mu^{<\theta}<\kappa$. By the Downward Lowenheim-Skölem theorem, find $X\prec Y$ with $D\subset X$ and $|X|=\kappa$. Since $D$ is dense in $X$, $d(X)<\kappa$.
\end{proof}

\begin{thm}\label{12a}
Let $\kappa\ge\w_1$. If there are cardinals $\mu$ and $\theta$ such that $\mu^{<\theta}=\mu<\kappa\le \mu^\theta$, then $\kappa$ is not $\kappa$-square compact.
\end{thm}

\begin{proof}
Immediate from Lemmas \ref{10} and \ref{12}.
\end{proof}

\begin{cor}\label{13}
Suppose $\lambda\ge\omega$. Then $\lambda^+$ is not $\lambda^+$-square compact.
\end{cor}

\begin{proof}
Let $\theta:=\min\{\nu:\lambda^\nu>\lambda\}$. By König's lemma, $\theta\le\cf(\lambda)$, so $\lambda^{<\theta}=\lambda$ by the minimality of $\theta$. Now apply Lemma \ref{12} with $\kappa=\lambda^+$ and $\mu=\lambda$.
\end{proof}
In particular, if $\kappa$ is $\kappa$-square compact, then $\kappa$ is a limit cardinal, hence weakly inaccessible. In fact, this can be improved:

\begin{cor}\label{14}
Suppose $\kappa$ is $\kappa$-square compact. Then $\kappa$ is strongly inaccessible.
\end{cor}

\begin{proof}
The fact that $\kappa$ is regular under the hypothesis was already mentioned in the introduction, and follows from \cite[Theorem 1]{hajnalSquarecompactCardinals1973}.

Suppose $\kappa$ is not strong limit. Let
\[
\theta=\min\{\nu:\exists\lambda\,(\nu\le\lambda<\kappa \le \lambda^\nu)\}.
\]
To see that this is well defined, fix $\delta<\kappa$ so that $2^\delta\ge\kappa$, and take $\lambda=\nu=\delta$.

Having fixed $\theta$, let $\lambda<\kappa$ be the least witness to the definition of $\theta$, that is $\theta \le \lambda < \kappa \le \lambda^\theta$ and $\lambda$ is least with these properties. Note that, if $\alpha<\theta$, then $\lambda^\alpha<\kappa$, since otherwise $\alpha$ contradicts the minimal choice of $\theta$.
\begin{claim}
$\theta$ is regular.
\end{claim}
\begin{cproof}
Suppose not, say $\theta^*=\cf(\theta)<\theta$. Fix $\seq{\theta_\xi:\xi<\theta^*}$ cofinal in $\theta$. By the minimality of $\theta$, $\lambda^{\theta_\xi}<\kappa$ for every $\xi<\theta^*$. Let $\lambda^*:=\sup\{\lambda^{\theta_\xi}:\xi<\theta^*\}$. Since $\kappa$ is regular and $\theta^*<\theta<\kappa$, it follows that $\lambda^*<\kappa$. We therefore have
\[
\kappa\le\lambda^\theta=\prod_{\xi<\theta^*}\lambda^{\theta_\xi}\le({\lambda^*})^{\theta^*}
\]
This contradicts the minimality of $\theta$.
\end{cproof}
Put $\mu=\lambda^{<\theta}$. Again, $\mu<\kappa$, because $\lambda^\alpha<\kappa$ for $\alpha<\theta$ and $\theta<\kappa=\cf(\kappa)$. Also, $\mu^\theta\ge\lambda^\theta\ge\kappa$. 

\begin{claim}
$\mu^{<\theta}=\mu$. 
\end{claim}

\begin{cproof}
We consider two separate cases.

\underline{Case 1:} $\alpha\mapsto\lambda^\alpha$ is eventually constant for $\alpha<\theta$. Note that this includes the case when $\theta$ is a successor cardinal. By definition of $\mu$, the eventual constant value must be $\mu$, so $\lambda^\alpha=\mu$ for all large enough $\alpha<\theta$. But then $\mu^\alpha=\mu$ whenever $\alpha<\theta$ is sufficiently big, hence $\mu^{<\theta}=\mu$.

\underline{Case 2:} $\lambda^\alpha$ is not eventually constant for $\alpha<\theta$. As $\theta$ is regular, $\cf(\mu)=\theta$. So, if $\alpha<\theta$, we have
\[
\mu^\alpha=\sum_{\beta<\theta}(\lambda^\beta)^\alpha=\mu.
\]
Therefore, $\mu^{<\theta}=\mu$.
\end{cproof}
Therefore, $\mu,\theta$ and $\kappa$ satisfy the conditions of Theorem \ref{12}, so $\kappa$ is not $\kappa$-square compact.
\end{proof}
\begin{thm}\label{15}
Let $\kappa$ be an uncountable cardinal. Then $\kappa$ is weakly compact if and only if it is $\kappa$-square compact.
\end{thm}

\begin{proof}
The forwards direction can be found in \cite[ Theorem 2]{hajnalSquarecompactCardinals1973}. The backwards direction is in \cite[Theorem 5.1]{buhagiarSquareCompactnessFilter2021}, under the additional hypothesis that $\kappa^{<\kappa}=\kappa$. But this is redundant when $\kappa$ is $\kappa$-square compact, because $\kappa$ is strongly inaccessible by Corollary \ref{14}.
\end{proof}

\nocite{kunenHandbookSettheoreticTopology1984,comfortTheoryUltrafilters1974}

To make the paper self-contained, we include a proof that strong compactness implies square compactness.

Recall that a \emph{subbase} for a topology $\tau$ (on a set $X$) is a family $\mathcal{S}$ such that $\tau$ is the smallest topology on $X$ including $\mathcal{S}$. Equivalently, the set of finite intersections of members of $\mathcal{S}$ is a base for $\tau$.

\begin{lem}\label{Alexander}
Let $\kappa$ be a strongly compact cardinal and $X$ a topological space. Suppose that there exists a subbase $\mathcal{S}$ such that for every cover of $X$ using members of $\mathcal{S}$ there exists a subcover of size $<\kappa$. Then $X$ is $\kappa$-compact.
\end{lem}

\begin{proof}
Let $\B$ be the collection of finite intersections of sets in $\mathcal S$, so $\B$ is a base for $X$. If suffices to argue that every open cover of $X$ consisting of members of $\B$ has a subcover of size $<\kappa$. Suppose towards a contradiction that this is not the case. Let $\U\subset\B$ be a cover of $X$ such that no subset of $\U$ of size $<\kappa$ covers $X$. Let $\I$ be the $\kappa$-complete ideal generated by $\U$:
\[
\I=\left\{A\subset X:\exists\, \mathcal{U}_0\in [\U]^{<\kappa}(A\subset \bigcup\U_0) \right\}
\]
By our assumption on $\U$, $X\not\in\I$, so $\I$ is proper. Since $\kappa$ is strongly compact, there is a prime $\kappa$-complete ideal $\J$ on $X$ such that $\I\subset\J$.

\begin{nclaim}
If $x\in X$ then there is some $W_x\in \J\cap\mathcal{S}$ such that $x\in W_x$.
\end{nclaim}

\begin{ncproof}
Fix $x\in X$. Since $\U\subset\B$ covers $X$, by definition of $\B$ there exists a finite sequence $\seq{W_i^x:i<n_x}\in \mathcal{S}^{n_x}$, where $n_x\in\w$, such that $x\in \bigcap_{i<n_x} W_i^x$. By definition of $\I$, $ \bigcap_{i<n_x} W_i^x\in\J $. Since $\J$ is prime, there exists $i_x<n_x$ such that $W_{i_x}^x\in\J$. Put $W_x=W_{i_x}^x$. This works.
\end{ncproof}

Using the claim we can choose, for each $x\in X$, a set $W_x\in \J\cap\mathcal{S}$ such that $x\in W_x$. Obviously, $\{W_x:x\in X\}$ covers $X$. Since $\{W_x:x\in X\}\subset\mathcal S$, our hypothesis on $\mathcal S$ implies the existence of $Y\in [X]^{<\kappa}$ such that $\{W_x:x\in Y\}$ covers $X$. In symbols, $X=\bigcup_{x\in Y}W_x$. But $W_x\in\J$, $\J$ is $\kappa$-complete and $|Y|<\kappa$, so $X\in\J$. This contradicts that $\J$ is a proper ideal.
\end{proof}

\begin{cor}[folklore]\label{strong}
Every strongly compact cardinal is square compact.
\end{cor}

\begin{proof}
Let $\kappa$ be strongly compact. Suppose $(X,\T)$ is a $\kappa$-compact space and let
\[
\mathcal{S}:=\{X\times U:U\in\T\}\cup \{U\times X:U\in\T\}.
\]
It is clear that $\mathcal{S}$ is a subbase for the product topology on $X^2$. 

Let $\U\subset{\mathcal{S}}$ cover $X$, we argue that $\U$ has a subcover of size $<\kappa$. By Lemma \ref{Alexander}, this is enough to complete the proof. Put $\U_0:=\U\cap (\T\times \{X\})$ and $\U_1=\U\cap(\{X\}\times \T)$, so that $\U=\U_0\cup \U_1$. Let $\V_0=\{V:V\times X\in \U_0\}$ and $\V_1=\{V:X\times V\in\V_1\}$.

\begin{nclaim}
At least one of $\V_0$ or $\V_1$ covers $X$.
\end{nclaim}

\begin{ncproof}
Suppose that $\bigcup\V_0\neq X$ and $\bigcup\V_1\neq X$. Pick $x_0\in X\setminus\bigcup\V_0$ and $x_1\in X\setminus \bigcup\V_1$. By assumption, $\U$ covers $X^2$, so $(x_0,x_1)\in U$ for some $U\in\U$. There are now two possibilities: either $U=V\times X$ for some $V\in\T$, in which case $x_0\in\bigcup\V_0$, or $U=X\times V$ for some $V\in\T$, in which case $x_1\in\bigcup\V_1$. In either case, we get a contradiction.
\end{ncproof}

Suppose that $\V_0$ covers $X$, the other case is analogous. Let $\V$ be a subcover of $\V_0$ of size $<\kappa$. Then $\{V\times X:V\in\V\}$ is a subcover of $\U$ of size $<\kappa$, which completes the proof.
\end{proof}

\section{Preliminaries on trees}

A \emph{chain} in a tree $T$ is a subset of $T$ which is linearly ordered by $<_T$. A \textit{branch} is a maximal chain. A \emph{cofinal branch} is a branch which meets every level of $T$.

Given a regular cardinal $\kappa$, we say that $T$ is a \textit{$\kappa$-tree} if and only if $\h(T)=\kappa$ and $|T_\alpha|<\kappa$ for every $\alpha<\kappa$. We say $T$ is a $\kappa$-\textit{Aronszajn tree} if and only if it is a $\kappa$-tree with no cofinal branches. An Aronszajn tree is just an $\aleph_1$-Aronszajn tree. Classically:

\begin{thm*}[König, \cite{konigUberSchlussweiseAus1927}]
There are no $\aleph_0$-Aronszajn trees.
\end{thm*}

\begin{thm*}[Aronszajn, see \cite{speckerProblemeSikorski1949}]
There is an Aronszajn tree. 
\end{thm*}

\begin{thm*}[Specker, \cite{speckerProblemeSikorski1949}]
If $\sf CH$ holds, then there is an $\aleph_2$-Aronszajn tree.
\end{thm*}

\begin{thm*}[Mitchell-Silver, \cite{mitchellAronszajnTreesIndependence1972}]
The theories
\begin{itemize}
\item $\sf ZFC$ + ``There is a weakly compact cardinal",
\item $\sf ZFC$ + ``There are no $\aleph_2$-Aronszajn trees"
\end{itemize}
are equiconsistent.
\end{thm*}

An \textit{antichain} in a tree is a set of pairwise incomparable elements of $T$. A $\kappa$-\textit{Suslin tree} is a $\kappa$-tree which has no chains or antichains of size $\kappa$. A \emph{Suslin tree} is an $\aleph_1$-Suslin tree.

A tree $T$ is \textit{normal} if and only if it satisfies the following conditions:

\begin{itemize}
\item It has a unique minimal element (called a \emph{root}),
\item for all $\alpha<\beta<\h(T)$ and all $x\in {T}_\alpha$ there is some $y\in{T}_\beta$ such that $x<y$,
\item for all $\alpha<\h(T)$ and $x\in {T}_\alpha$ there exist $y,z\in T$ such that $x<y$, $x<z$, and $y\perp z$.
\end{itemize}

\begin{lem}[folklore]\label{chaanc}
Let $T$ be a normal $\kappa$-tree. If $T$ is has no antichains of size $\kappa$, then $T$ is $\kappa$-Suslin. 
\end{lem}

\begin{proof}
If $b$ is a branch through $T$ of length $\kappa$, use the normality of $T$ to pick, for each $x\in b$, some $y_x\in I(x)$ such that $y_x\not\in b$. Then $A=\{y_x:x\in b\}$ is an antichain and $|A|=\kappa$.
\end{proof}

In view of Lemma \ref{chaanc}, to check whether a given normal $\kappa$-tree is Suslin, one ``only" has to argue that all of its antichains have size less than $\kappa$. We shall make use of this fact without any further mention.

We also recall that an $\aleph_1$-tree is \emph{special} if and only if it can be written as a countable union of antichains. Equivalently, $T$ is special if and only if there is an order preserving map $T\to\Q$, see \cite[Lemma III.5.17]{kunenSetTheory2011}.

Let $(T,<)$ be a tree. If $X\subset T$, we let $\uarrow X:=\{y\in T:\exists x\in X(x\le y)\}$. If $X=\{x\}$, we write $\uarrow x$ instead of $\uarrow\{x\}$. The symbols $\darrow X$ and $\darrow x$ are defined analogously.

\section{The fine wedge topology}

If $T$ is a tree, the \textit{fine wedge topology} on $T$ is generated by the sets $\uarrow t$ and their complements, where $t\in T$.

Note that, if $x<y$, then $(\uarrow x) \setminus \uarrow y$ and $\uarrow y$ are disjoint open neighbourhoods of $x$ and $y$, respectively. If $x\perp y$, then $\uarrow x$ and $\uarrow y$ are disjoint open neighbourhoods of $x$ and $y$. We have thus shown that the topology is Hausdorff.

All topological notions below refer to the fine wedge topology.

If $T$ is \emph{finitely splitting at $x$} (that is $|I(x)|<\aleph_0$), then the identity
\[
\{x\}=(\uarrow x)\cap \bigcap_{y\in I(x)}(\uarrow y)^c
\]
shows that $x$ is isolated. Therefore, if $T$ is finitely splitting, the fine wedge topology is just the discrete topology on $T$. The interplay between finite and infinitely splitting trees will play a key role in our work, see Theorem \ref{subtree}.

Recall that, if $X$ is a topological space and $x\in X$, we say that a collection of open sets $\mathcal B$ is a \emph{local base at $x$} if and only if for every open set $U$ with $x\in U$ there is some $B\in \mathcal B$ with $x\in B\subset U$. We define the \emph{character of $x$} to be the cardinal $\chi(x,X):=\min\{|\mathcal B|:\mathcal{B} \text{ is a local base at }x\}$.

\begin{lem}\label{base}
Let $T$ be a tree. Given $x\in T$, the sets
\[
(\uarrow x)\setminus \uarrow F,
\]
where $F\in [I(x)]^{<\w}$, form a local base at $x$, and so $\chi(x,T)=|I(x)|$. In particular, if every node has $\aleph_0$ many immediate successors, then the fine-wedge topology is first countable.
\end{lem}

\begin{proof}
Let $U$ be a basic open neighbourhood of $x$, say
\[
x\in U=\bigcap_{i<n}\uarrow x_i\setminus \bigcup_{j<m}\uarrow y_j
\]
for some $x_i,y_j\in T$, $n,m\in\w$. Let $J=\{j<m:x< y_j\}$. For each $j\in J$, let $z_j\in I(x)$ be the unique point with $z_j\le y_j$. Then
\[
x\in (\uarrow x)\setminus \bigcup_{j\in J}\uarrow z_j\subset U,
\]
which completes the proof.
\end{proof}  

\begin{lem}
Suppose $\kappa$ is a regular cardinal. Let $(T,<)$ be a $\kappa$-tree which is $\kappa$-compact. Then $T$ is $\kappa$-Aronszajn.
\end{lem}

\begin{proof}
Suppose that $b$ is a cofinal branch through $T$. Then $\{(\uarrow x)^c:x\in b\}$ has no subcover of size $<\kappa$ by regularity.
\end{proof}

\begin{lem}
Let $(T,<)$ be a $\kappa$-Aronszajn tree. Then every cover of $T$ by subbasic open sets has a subcover of size $<\kappa$.
\end{lem}

\begin{proof}
Let $\mathcal U$ be a cover of $T$ by subbasic open sets. Observe that, if all nodes at some level of the tree belong to a cone from $\mathcal U$, then these cones give a subcover of size $<\kappa$ of the tree above that level. But there are less than $\kappa$ many nodes below that level, so we're done. The idea is essentially to show that such a ``good" level must exist.

Consider the sets
\[
A=\{t\in T:\uarrow t\in \mathcal{U}\}
\]
and
\[
B=\{t\in T:(\uarrow t)^c\in\mathcal{U}\},
\]
where ${}^c$ denotes complementation with respect to $T$. Suppose there are $s,t\in B$ such that $s\perp t$. Then $(\uarrow s)\cap (\uarrow t)=\emptyset$, so $(\uarrow s)^c\cup (\uarrow t)^c=T$ and we've found a finite subcover. Therefore, we may assume that $B$ is linearly ordered, hence a branch. Put
\[
X=\bigcap_{x\in B}\uarrow x
\]
and observe that
\[
T=X\cup\bigcup_{x\in B}(\uarrow x)^c.
\]
Since the tree is $\kappa$-Aronszajn, $|B|<\kappa$, hence we only need to show that $X$ is covered by some subset of $\mathcal U$ of size less than $\kappa$. Let $\alpha$ be the least height of a member of $X$ (if $X$ is empty, there's nothing to do), and pick $y\in X\cap{T}_\alpha$. Since $\mathcal{U}$ covers $T$, there is some $U\in \mathcal U$ with $y\in U$. If $U=(\uarrow t)^c$ for some $t$, then $t\in B$, so $t<y$ because $y\in X$, contradicting that $y\in(\uarrow t)^c$. It follows that $y\in \uarrow t$ for some $t\in A$. This shows that ${T}_\alpha\cap X$ is one of the good levels described in the first paragraph of the proof, and we're done.
\end{proof}

We isolate the following elementary result from point-set topology:
\begin{lem}
Let $X$ be a topological space and $\kappa$ an infinite cardinal. Suppose we have a sequence $\seq{\mathcal{B}_x:x\in X}$ such that $\B_x$ is a local base at $x$ for every $x\in X$. Then $X$ is $\kappa$-compact if and only if for each $\Gamma\in\prod_{x\in X}\B_x$ there is some $Y\in [X]^{<\kappa}$ such that $X=\bigcup_{y\in Y} \Gamma(y)$. 
\end{lem}
\begin{proof}
The forwards direction is easy. We prove the backwards implication. Let $\U$ be an open cover of $X$. Choose, for each $x\in X$, some $U_x\in\U$ such that $x\in U_x$. Given $x\in X$, we know that $\B_x$ is an open base at $x$, hence we can find $\Gamma(x)\in \B_x$ such that $x\in\Gamma(x)\subset U_x$. This defines $\Gamma\in\prod_{x\in X}\B_x$. By assumption, there is some $Y\in [X]^{<\kappa}$ such that $X=\bigcup_{y\in Y}\Gamma(y)$. But then $X=\bigcup_{y\in Y}U_y$ and $|\{U_y:y\in Y\}|<\kappa$.
\end{proof}

In the tree context, we'll be looking at the system of local bases formed by the sets $\uarrow x\setminus\uarrow f(x)$, where $f\in \prod_{x\in T}[I(x)]^{<\w}$. We shall say that such an $f$ \emph{codes} the cover $\U_f:=\{\uarrow x\setminus \uarrow f(x):x\in T\}$. Going forward, we will blur the distinction between the function $f$ and the open cover $\U_f$, and speak simply of the \emph{cover} $f$. We will also consider $\U_f$ for $f\in\prod_{x\in X}[I(x)]^{<\w}$, where $X\subset T$ (of course, $\U_f$ might not cover $T$).

Note that covers of this kind have the following important property:

\begin{lem}\label{2.2}
Let $T$ be a tree and $f\in\prod_{x\in T}[I(x)]^{<\w}$. The following are equivalent:
\begin{enumerate}[font=\normalfont]
\item $f$ has a countable subcover.
\item There is a limit ordinal $\alpha<\w_1$ such that $f\restr (T\restr\alpha)$ covers $T$.
\item There is an ordinal $\alpha<\w_1$ such that for every $x\in {T}_\alpha$ there is some $y\in T\restr\alpha$ such that $x\in\uarrow y\setminus\uarrow f(y)$.
\end{enumerate}
\end{lem}

\begin{proof}
Trivial.
\end{proof}

\begin{defi}
Let $T$ be an $\aleph_1$-tree and $f\in\prod_{x\in T} [I(x)]^{<\w}$. We say that a point $x\in T$ is \textit{safe} (for $f$) if and only if for all $y<x$, $x\in \uarrow f(y)$.
\end{defi}
An immediate consequence of the definition is:
\begin{lem}
Let $T$ be an $\aleph_1$-tree and $f\in\prod_{x\in T} [I(x)]^{<\w}$. If $x\in T$ is safe, then so is every $y<x$. Also, if $z\in I(x)$ (with $x$ safe), then $z$ is safe if and only if $z\in f(x)$.
\end{lem}

\begin{proof}
Trivial.
\end{proof}

The key property of safe points is the following:

\begin{lem}\label{safe}
Let $T$ be an $\aleph_1$-tree and $f\in\prod_{x\in T}[I(x)]^{<\w}$. The following are equivalent:
\begin{enumerate}[font=\normalfont]
\item $f$ has no countable subcover.
\item For every $\alpha<\w_1$ there is a safe point of height $\alpha$.
\item The set $\{\h(x):x\text{ is safe}\}$ is unbounded in $\w_1$.
\end{enumerate}
\end{lem}

\begin{proof}
$(1)\Rightarrow (2)$: Fix $\alpha<\w_1$. By \ref{2.2}, there is some $x\in{T}_\alpha$ such that for all $y\in T\restr\alpha$, $x\not\in \uarrow y\setminus \uarrow f(y)$. In particular, if $y<x$, $x\in \uarrow f(y)$, so $x$ is safe.

$(2)\Rightarrow (3)$: Trivial.

$(3)\Rightarrow (1)$: Suppose towards a contradiction that $f$ has a countable subcover. By Lemma \ref{2.2}, there is some limit $\gamma<\w_1$ such that $f\restr (T\restr\gamma)$ covers $T$. Choose a safe point $x$ with $\gamma<\h(x)$. Let $y\in T\restr\gamma$. If $y\not< x$, then obviously $x\not \in\uarrow y\setminus \uarrow f(y)$. If $y<x$, the safety of $x$ implies that $x\in \uarrow f(y)$, so $x\not\in \uarrow y\setminus \uarrow f(y)$. In either case, $x\not\in\bigcup\U_{f\restr (T\restr\gamma)}=T$, contradiction.
\end{proof}

Recall that a subtree of a tree is a downwards closed subset. Note that, if $S$ is a subtree of $T$ and $\alpha<\h(T)$, then $S_\alpha=S\cap T_\alpha$. Also, if $x\in S$, then $I_S(x)=I_T(x)\cap S$.

\begin{thm}\label{subtree}
Let $T$ be an $\aleph_1$-tree. Then $T$ is Lindelöf if and only if every finitely splitting subtree of $T$ is countable.
\end{thm}

\begin{proof}
$\Rightarrow)$ Suppose $S\subset T$ is a finitely splitting subtree of $T$ with $\h(S)=\aleph_1$. Define $f(x)=I_T(x)\cap S$, where $x\in T$. We claim that the cover coded by $f$ has no countable subcover. To see this, fix a limit ordinal $\alpha<\w_1$. Choose $x\in S\cap T_\alpha$ and $y<x$ (in $T$). Since $S$ is a subtree, $y\in S$. Let $z$ be the unique element of $I_T(y)\cap \darrow x$. As $z\le x\in S$, we see that $z\in S$, hence $z\in f(y)$ and so $x\in\uarrow f(y)$. This shows that $x$ is safe. There is therefore a safe point at every limit level, so we're done by Lemma \ref{safe}.

$\Leftarrow$) Let $f$ code a cover with no countable subcover. Let $S$ be the set of points of $S$ which are safe for $f$. It is clear that $S$ is a subtree of $T$. Given $x\in S$, we see that $I_{S}(x)=f(x)$, so $S$ is finitely splitting. Since $f$ has no countable subcover, $\h(S)=\aleph_1$.
\end{proof}

We end this section by showing that a Suslin tree is automatically Lindelöf. To do this, we need to recall some elementary facts on forcing with Suslin trees. Given a tree $T$, let ${\mathbb{P}}_T$ be $T$ upside down. More precisely, the underlying set of ${\mathbb{P}}_T$ is (the underlying set of) $T$ and the order on ${\mathbb{P}}_T$ is the dual order on $T$.

\begin{lem}\label{Sf}
Let $T$ be a Suslin tree. Then ${\mathbb{P}}_T$ is ccc and, for every dense open set $D\subset {\mathbb{P}}_T$ there is some $\alpha<\w_1$ such that $\{x\in T:\h(x)\ge \alpha\}\subset D$. In particular, ${\mathbb{P}}_T$ is \emph{countably distributive} (every countable intersection of dense open sets is dense).
\end{lem}

\begin{proof}
The poset ${\mathbb{P}}_T$ being ccc is just a restatement of $T$ having no uncountable antichains. To prove the remaining statements, let $D\subset {\mathbb{P}}_T$ be dense and open. Let $A\subset D$ be a maximal antichain, which must therefore be countable. Fix $\alpha<\w_1$ such that $A\subset T\restr\alpha$. If $x\in T\restr [\alpha,\w_1)$, there is some $a\in A$ with $a\parallel x$, hence $a<x$. Since $D$ is open, $x\in D$.
\end{proof}

\begin{lem}\label{gen}
Let $T$ be a Suslin tree in the universe $V$. Let $W$ be an outer model of $V$. Suppose that $b\in W$ is a cofinal branch through $T$. Then $b$ is ${\mathbb{P}}_T$-generic over $V$.
\end{lem}

\begin{proof}
Let $D\in V$ be dense and open in ${\mathbb{P}}_T$. Choose $\alpha<\w_1^V$ with $T\restr [\alpha,\w_1)\subset D$. Then any point in $b$ of height at least $\alpha$ must belong to $D$.
\end{proof}

\begin{thm}\label{fbS}
Let $T$ be an infinitely splitting Suslin tree. Then every finitely splitting subtree of $T$ is countable.
\end{thm}

\begin{proof}
Suppose that $S$ is a finitely splitting subtree of $T$ with height $\aleph_1$. Since $T$ is Suslin, so is $S$. Let $G$ be ${\mathbb{P}}(S)$-generic over $V$. By \ref{gen}, $G$ is ${\mathbb{P}}_T$-generic over $V$. But an easy density argument shows that every ${\mathbb{P}}_T$-generic branch is disjoint from $S$ above some node of $T$, because $S$ is finitely splitting and $T$ is not. This is a contradiction.
\end{proof}

\begin{cor}\label{S->L}
Every Suslin tree is Lindelöf.
\end{cor}

\begin{proof}
Immediate from theorems \ref{subtree} and \ref{fbS}.
\end{proof}

\section{Examples of Lindelöf and non-Lindelöf trees}

If $s$ and $t$ are two functions with domain $\alpha$, we let $\Delta(s,t):=\{\xi<\alpha:s(\xi)\neq t(\xi)\}$. We write $s=^*t$ if and only if $\Delta(s,t)$ is finite.

We shall say that $\seq{e_\alpha:\alpha<\w_1}$ is a \emph{coherent sequence of injections} if $e_\alpha:\alpha\to\w$ is injective and $e_\alpha=^*e_\beta\restr\alpha$ for all $\alpha,\beta<\w_1$. Note that each $e_\alpha$ must have coinfinite range. Going forward, we shall speak simply of coherent sequences, which in the literature usually refers to finite to one functions. Let
\[
T^{\vec e}=\bigcup_{\alpha<\w_1}\{s\in {}^\alpha\w: s \text{ is injective}\,\wedge s=^*e_\alpha\}.
\]
It is clear that $T^{\vec e}$ is an infinitely splitting Aronszajn tree.

Our next goal is to construct a non-Lindelöf Aronszajn tree. The idea is to build a finitely splitting Aronszajn tree, and then make $\aleph_0$ many new nodes ``sprout" at each node, producing a larger, infinitely splitting tree. The constructions of Aronszajn trees that we're familiar with all produce infinitely splitting trees, so our first task is to build a finitely splitting one.

Recall that a tree if \emph{splitting} if every node has at least two distinct immediate successors. 

\begin{lem}\label{finAZFC}
There is a splitting subtree of $2^{<\w_1}$ which is Aronszajn.
\end{lem}

\begin{proof}
Fix a coherent sequence $\vec e$ and a bijection $f:\w_1\times\w\to\w_1$ such that for every limit ordinal $\gamma<\w_1$, $f[\gamma\times\w]=\gamma$. Put $\Gamma=\lim(\w_1)\cup\{0\}$. Define $x_\alpha\in 2^\alpha$ for $\alpha\in\Gamma$ by $x_\alpha=\chi_{f[e_\alpha]}$, where $\chi_A$ denotes the characteristic function of $A$. This makes sense because $f[e_\alpha]\subset \alpha$.

\begin{nclaim}
If $\alpha,\beta\in\Gamma$ and $\alpha<\beta$ then $x_\alpha=^* x_\beta\restr\alpha$.
\end{nclaim}

\begin{ncproof}
For readability, we extend our $\Delta(s,t)$ notation to allow functions with different domains. More precisely, if $\dom(s)\le\dom(t)$, we write $\Delta(s,t)$ for $\Delta(s,t\restr\dom(s))$.

If $\alpha=0$ everything is trivial so we assume that $\alpha\ge\w$. Fix $\eta\in\Delta(x_\alpha,x_\beta)$, so $x_\alpha(\eta)\neq x_\beta(\eta)$. Since $f[\alpha\times\w]=\alpha$, there exist unique $\xi<\alpha$ and $n\in\w$ with $f(\xi,n)=\eta$.

Since $x_\alpha(\eta)\neq x_\beta(\eta)$, there are two possibilities: $\eta\in f[e_\alpha]\setminus f[e_\beta]$ or $\eta\in f[e_\beta]\setminus f[e_\alpha]$. We consider the two cases separately:
\begin{enumerate}[1.]
\item Suppose that $\eta\in f[e_\alpha]\setminus f[e_\beta]$. Then $e_\alpha(\xi)=n$ but $e_\beta(\xi)\neq n$, so $\xi\in\Delta(e_\alpha,e_\beta)$ and $\eta=f(\xi,e_\alpha(\xi))$.
\item Suppose that $\eta\in f[e_\beta]\setminus f[e_\alpha]$. Then $e_\beta(\xi)=n$ but $e_\alpha(\xi)\neq n$, so $\xi\in\Delta (e_\alpha,e_\beta)$ and $\eta=f(\xi,e_\beta(\xi))$.
\end{enumerate}
We have therefore established that
\[
\Delta(x_\alpha,x_\beta)\subset \{f(\xi,e_\alpha(\xi)):\xi\in\Delta(e_\alpha,e_\beta)\}\cup \{f(\xi,e_\beta(\xi)):\xi\in\Delta(e_\alpha,e_\beta)\}.
\]
Since $\vec e$ is coherent, the union on the right is finite, hence $\Delta(x_\alpha,x_\beta)$ is finite too.
\end{ncproof}

Given $\alpha<\w_1$, let $\gamma_\alpha$ be the unique ordinal in $\Gamma$ such that $\gamma_\alpha\le\alpha<\gamma_\alpha+\w$. Note that $\alpha\in\Gamma$ if and only if $\alpha = \gamma_\alpha$. Define
\[
T=\bigcup_{\alpha<\w_1}\{x\in 2^\alpha: x\restr \gamma_\alpha=^*x_{\gamma_\alpha}\}
\]
By the claim, $T$ is a subtree of $2^{<\w_1}$. Indeed, if $x\in 2^\alpha$ and $y\in T\cap 2^\beta$ satisfy $x\subset y$, then $\gamma_\alpha\le \gamma_\beta$, so $x\restr \gamma_\alpha=(y\restr \gamma_\beta)\restr \gamma_\alpha=^*x_{\gamma_\beta}\restr\gamma_\alpha=^* x_{\gamma_\alpha}$, where the last $=^*$ follows from the claim. The fact that $T$ has countable levels is immediate from the claim. As $x_\alpha\in T_\alpha$ for every $\alpha\in\Gamma$, we see that $T$ has height $\w_1$, and so $T$ is an $\aleph_1$-tree.

We point out that that, if $x\in T$, then $x^\frown \seq{0},x^\frown\seq{1}\in T$, because $\gamma_\alpha=\gamma_{\alpha+1}$ for every $\alpha<\w_1$.

To see that $T$ is Aronszajn, assume towards a contradiction that $\seq{y_\alpha:\alpha<\w_1}$ is a branch through $T$, so in particular $y_\alpha=^*x_\alpha$ for every $\alpha\in\lim(\w_1)$. Find $\eta_\alpha<\alpha$ for each $\alpha\in\lim(\w_1)$ so that $\Delta(y_\alpha,x_\alpha)\subset \eta_\alpha$. Since $\alpha\mapsto\eta_\alpha$ is regressive, by Fodor's Lemma there is some stationary set $E\subset\lim(\w_1)$ and some $\eta<\w_1$ such that $\eta_\alpha=\eta$ for every $\alpha\in E$. Since $|T_\eta|\le\aleph_0$, we may find $E'\subset E$ stationary and $x,y\in 2^\eta$ such that $y_\alpha\restr\eta=y$ and $x_\alpha\restr\eta=x$ for every $\alpha\in E'$. If $\alpha,\beta\in E'$, $\alpha<\beta$,
\[
x_\alpha\restr [\eta,\alpha)=y_\alpha\restr [\eta,\alpha)=y_\beta\restr [\eta,\alpha)=x_\beta\restr [\eta,\alpha)
\]
by the choice of $\eta$. Since $x_\alpha\restr\eta=x=x_\beta\restr\eta$, it follows that $x_\alpha=x_\beta\restr\alpha$. Finally, if $\xi<\alpha$ and $n:=e_\alpha(\xi)$, then $x_\alpha(f(\xi,n))=1$, so $x_\beta(f(\xi,n))=1$, so $f(\xi,n)\in f[e_\beta]$, so $e_\beta(\xi)=n$. We have thus shown that $\seq{e_\alpha:\alpha\in E'}$ is a chain, which is absurd because $E'$ is uncountable.
\end{proof}

\begin{lem}\label{finintoinf}
There is an infinitely splitting Aronszajn tree with a finitely splitting subtree of height $\aleph_1$.
\end{lem}

\begin{proof}
Let $\vec e$ be a coherent sequence and let $T\subset 2^{<\w_1}$ be the tree constructed from $\vec e$ in Lemma \ref{finAZFC}. We recursively define a tree $U\subset\w^{<\w_1}$ level by level, starting with $U_0=\{\emptyset\}$. For successor stages, we let $U_{\alpha+1}=\{u^\smallfrown \seq{n}: u\in U_\alpha\wedge n\in\w\}$. If $\alpha<\w_1$ is a limit ordinal, we let $U_\alpha=\{u\cup t\restr [\dom(u),\alpha):u\in U\restr \alpha \wedge t \in T_\alpha\}$. An easy induction shows that $U_\alpha$ is countable and that $T_\alpha\subset U_\alpha$ for every $\alpha$.

\begin{claim}
If $\beta<\alpha$, $t\in T_\alpha$ and $u\in U_\beta$, then $u\cup t\restr[\beta,\alpha)\in U_\alpha$.
\end{claim}

\begin{cproof}
By induction on $\alpha$:
\begin{itemize}
\item If $\alpha=0$, then it's obvious.
\item Suppose that $\alpha$ is a successor ordinal, say $\alpha=\gamma+1$. Then $ u \cup (t\restr [\beta,\alpha) )= (u\cup t\restr [\beta,\gamma))^\smallfrown \seq{t(\gamma)} $ which belongs to $U_\alpha$ by the inductive hypothesis and the definition of $U_\alpha$.
\item If $\alpha$ is a limit ordinal, this is immediate from the definition of $U_\alpha$.\qedhere
\end{itemize} 
\end{cproof}

\begin{claim}
If $\alpha<\w_1$, $\beta<\alpha$ and $v\in U_\alpha$, then $v\restr\beta\in U$.
\end{claim}

\begin{cproof}
By induction on $\alpha$:
\begin{itemize}
\item If $\alpha=0$ then it's vacuously true.
\item Suppose that $\alpha$ is a successor ordinal, say $\alpha=\gamma+1$. Fix $v\in U_\alpha$, so $v=u^\smallfrown \seq{n}$ for some $u\in U_{\gamma}$ and $n\in\w$. Let $\beta<\alpha$. If $\beta=\gamma$, then $v\restr\beta=u\in U$ as desired. If $\beta<\gamma$, then $v\restr\beta=u\restr\beta\in U$ by the inductive hypothesis applied to $\gamma$.
\item Suppose that $\alpha$ is a limit ordinal and let $v\in U_\alpha$, say $v=u\cup (t\restr [\dom(u),\alpha))$, where $u\in U_\gamma$, $\gamma<\alpha$, and $t\in T$. Let $\beta<\alpha$. If $\beta=\gamma$, then $v\restr\beta=u\in U_\beta$. If $\beta<\gamma$, then $v\restr\beta=u\restr\beta\in U$ by the inductive hypothesis applied to $\gamma$. If $\gamma<\beta$, then $v\restr\beta = u\cup (t\restr[\dom(u),\beta))$. We now induct on $\beta$. If $\beta$ is a limit ordinal, then $v\restr\beta\in U_\beta$ by definition of $U_\beta$. If $\beta$ is a successor ordinal, say $\beta=\delta+1$, then $v\restr\beta = (u\cup t\restr [\dom(u),\delta))^\smallfrown \seq{t(\delta)} $. By Claim 1, $ u\cup t\restr [\dom(u),\delta) \in U_\delta $, and so $v\restr\beta\in U_\beta$ by definition of $U_\beta$.\qedhere 
\end{itemize}
\end{cproof}

To see that $U$ is Aronszajn, suppose towards a contradiction that $b$ is a cofinal branch through $U$. Let $x_\alpha$ be the $ \alpha^\text{th}$ point of $b$. For each limit ordinal $\alpha<\w_1$, there exist $u_\alpha<x_\alpha$ and $t_\alpha\in T_\alpha$ such that $x_\alpha=u_\alpha\cup t_\alpha\restr [\dom(u_\alpha),\alpha)$. Define  $f:\lim(\w_1)\to\w_1$ by $f(\alpha)=\dom(u_\alpha)$. Since $f$ is regressive, we may find a stationary set $\Gamma\subset\lim(\w_1)$ and some $\eta<\w_1$ such that $f``\Gamma=\{\eta\}$. As $|T_\eta|\le\aleph_0$, we may assume, by shrinking $\Gamma$ if necessary, that there is some $t\in T_\eta$ such that $t=t_\alpha\restr\eta$ for all $\alpha\in \Gamma$. If $\alpha,\beta\in\Gamma$, $\alpha<\beta$, then $t_\alpha\restr[\eta,\alpha)=x_\alpha\restr[\eta,\alpha)=x_\beta\restr[\eta,\alpha)=t_\beta\restr[\eta,\alpha)$ by our choice of $\eta$. But then $t_\alpha=t_\beta\restr\alpha$ because they both agree with $t$ below $\eta$. This means that $\seq{t_\alpha:\alpha\in\Gamma}$ is an uncountable chain in $T$, contradicting that $T$ is Aronszajn.
\end{proof}

\begin{cor}
There is a non-Lindelöf Aronszajn tree.
\end{cor}

\begin{proof}
Apply lemmas \ref{finAZFC} and \ref{finintoinf} to obtain an infinitely splitting Aronszajn tree $T$ with a finitely splitting subtree of uncountable height. Then $T$ is not Lindelöf by  Theorem \ref{subtree}.
\end{proof}
Recall that Jensen's \emph{diamond principle}, denoted $\diamond$, asserts the existence of a sequence $\seq{A_\alpha:\alpha<\w_1}$ such that $A_\alpha\subset\alpha$ and for every $A\subset\w_1$ there exist stationarily many $\alpha<\w_1$ with $A\cap\alpha=A_\alpha$. We shall need to modify the $\diamond$-sequence so that it ``guesses" functions $\w_1\to[\w_1]^{<\w}$.

\begin{lem}\label{finDiamond}
The principle $\diamondsuit$ holds if and only if there is a sequence $\seq{f_\alpha:\alpha<\w_1}$ such that $f_\alpha:\alpha\to[\alpha]^{<\w}$ and for every $f:\w_1\to[\w_1]^{<\w}$ the set $\{\alpha:f\restr \alpha=f_\alpha\}$ is stationary.
\end{lem}

The proof is a standard coding argument, and we omit it.

A sequence $\seq{f_\alpha:\alpha<\w_1}$ of the kind appearing in the statement of Lemma \ref{finDiamond} will also be referred to as a $\diamondsuit$-sequence.
\begin{thm}\label{spL}
If $\diamondsuit$ holds, then there is a special Lindelöf tree.
\end{thm}

\begin{proof}
We construct a \textit{normal}, infinitely splitting $T$ level by level, together with a specializing function $\phi:T\to\Q$. To make sure that $\phi$ can be extended at limit stages, we require that
\begin{equation}\label{s1}
\forall x\in T\restr\alpha\, \forall q\in\Q(q<\phi(x)\to \exists y\in {T}_\alpha(x<y\wedge \phi(y)=q))\tag{{$\ast$}}
\end{equation}
holds for every $\alpha<\w_1$. The construction will also depend on a fixed $\diamondsuit$ sequence $\seq{f_\alpha:\alpha<\w_1}$ such that $f_\alpha:\alpha\to[\alpha]^{<\w}$ and, for every $f:\w_1\to [\w_1]^{<\w}$, the set $\{\alpha<\w_1:f\restr\alpha=f_\alpha\}$ is stationary. Such a sequence of functions exists by Lemma \ref{finDiamond}.

The underlying set of the tree will be $\w_1$, with $T\restr\alpha=\w\alpha$ and ${T}_\alpha=[\w\alpha,\w(\alpha+1))$ for $\alpha>\w$. Let $0$ be the root of $T$. For the successor step, given a node at level $\alpha$ we put $\aleph_0$ many nodes immediately above $x$ and let $\phi\restr I(x)$ be a bijection between $I(x)$ and $\Q\cap (\phi(x),\infty)$. Note that normality and \eqref{s1} continue to hold.

Suppose now that $\alpha<\w_1$ is a limit ordinal and we've constructed $T\restr\alpha$ and $\phi\restr T\restr\alpha$. List the set $\{(x,q)\in T\restr\alpha\times\Q:\phi(x)<q\}$ as $\{(x_k,q_k):k\in\w\}$. 
The construction splits into two cases. 

\underline{Case 1:} $\w\alpha>\alpha$. Fix $k\in\w$. By normality and \eqref{s1}, we can choose a branch $b_k$ with least element $x_k$ such that the heights of members of $b_k$ converge to $\alpha$ and $\sup(\phi``b_k)=q_k$. Now put a node above $b_k$ and let $\phi$ take the value $q_k$ at this node.

\underline{Case 2:} $\w\alpha=\alpha$. Fix $k\in\w$. Choose a cofinal branch $b_k\subset T\restr\alpha$ satisfying the following properties:
\begin{enumerate}[(i), font=\normalfont]
\item $x_k\in b_k$,
\item $\sup(\phi``b_k)=q_k$,
\item the unique point on $b_k$ immediately above $x_k$ does not belong to $f_\alpha(x_k)$.
\end{enumerate}
To achieve this, use that, by our construction at successor stages, $\phi$ maps $I(x_k)$ onto $\Q\cap (\phi(x_k),\infty)$ to find a point $z_k\in I(x_k)\setminus f_\alpha(x_k)$ with $\phi(z_k) < q_k$. Then construct $b_k$ by repeatedly applying $(\ast)$ and taking a downwards closure. Finally, put a node above $b_k$ on level $\alpha$ and let $\phi$ take the value $q_k$ at this node. This completes the construction of $T$.

It is obvious that $\phi$ is a specializing function for $T$. To see that $T$ is Lindelöf, consider a basic cover $f\in \prod_{x\in T}[I(x)]^{<\w}$. Pick $\alpha<\w_1$ limit such that $\w\cdot\alpha=\alpha$ and $f\restr\alpha=f_\alpha$, so that $f\restr (T\restr\alpha)=f_\alpha$.

The key observation is that every node at level $\alpha$ is in $\uarrow y\setminus\bigcup_{z\in f_\alpha(y)}\uarrow z$ for some $y\in T\restr\alpha$. By Lemma \ref{2.2}, $f_\alpha$ covers $T$, hence so does $f\restr\alpha$, and we're done.
\end{proof}

\begin{cor}\label{LnonS}
If $\diamondsuit$ holds, then there is a Lindel\"of tree that is not Suslin.
\end{cor}

\begin{proof}
Assume $\diamondsuit$ holds. By Theorem \ref{spL}, there is a special Lindel\"of tree. But special trees can never be Suslin.
\end{proof}
So, under $\diamondsuit$, we have strict inclusion of our classes of trees:
\[
\{\text{Suslin}\}\subsetneq \{\text{Lindelöf}\} \subsetneq \{\text{Aronszajn}\}.
\]

\section{Adding subtrees}

In this section, we address the following question: given an infinitely splitting Aronszajn tree $T$, can we find a poset $\mathbb{P}$ such that $\Vdash_\mathbb{P}$ $T$ is a non-Lindel\"of Aronszajn tree? In other words, $\mathbb{P}$ forces the existence of an uncountable finitely splitting subtree of $T$ but at the same time adds no uncountable branches to $T$.

\begin{defi}
Given an infinitely splitting $\aleph_1$-tree $T$, we let $\D_T$ be the following poset: conditions are finite functions with $p\in \prod_{x\in F}[I(x)]^{<\w}$, where $F\in [T]^{<\w}$, such that $\emptyset\not\in\ran(p)$ and which satisfy the following property:
\begin{equation}
\forall x,y\in\dom(p)(x<y\to y\restr(\h(x)+1)\in p(x)).\tag{$\dagger$}
\end{equation}
The order on $\D_T$ is $p\le q \iff p\supset q$.
\end{defi}
The idea is that a condition is a promise that a certain subtree will be finitely splitting at each point of the condition's domain. 
\begin{lem}[Baumgartner]\label{Baumgartner}
Let $T$ be a tree with no uncountable branches. Suppose that $\mathscr{A}$ is an uncountable collection of pairwise disjoint non-empty finite subsets of $T$. Then there exist $a,b\in \mathscr{A}$ such that for all $x\in a$ and $y\in b$, $x\perp y$.
\end{lem}

For a proof, see \cite[Lemma III.5.18]{kunenSetTheory2011}.

\begin{lem}\label{ccc}
Let $T$ be an Aronszajn tree. Then $\D_T$ has the ccc.
\end{lem}

\begin{proof}
Let $\seq{p_\alpha:\alpha<\w_1}\subset\D_T$. Letting $d_\alpha=\dom(p_\alpha)$ and thinning out if necessary, we may assume without loss of generality that $\{d_\alpha:\alpha<\w_1\}$ forms a $\Delta$-system with root $R$ and that $\alpha<\beta<\w_1$ implies $p_\alpha\restr R=p_\beta\restr R$. Apply Lemma \ref{Baumgartner} to find $\alpha<\beta<\w_1$ such that $x\perp y$ for all $x\in d_\alpha\setminus R$ and $y\in d_\beta\setminus R$. Let $r=p_\alpha \cup p_\beta$. Then $r$ is a condition, because the only way two points in its domain are comparable is if they both belong to either $d_\alpha$ or $d_\beta$, hence $(\dagger)$ holds. Since $r\le p_\alpha,p_\beta$, we are done.
\end{proof}

\begin{lem}\label{dandelion}
Let $T$ be an infinitely splitting Aronszajn tree and $\dot S$ a $\D_T$ name for the set $\bigcup_{p\in\dot G}\dom(p)$, where $\dot G$ is a $\D_T$-name for the generic filter.
\begin{enumerate}[(i),font=\normalfont]
\item For every $p\in\D_T$, every $y\in\dom(p)$ and every $x\in T$, if $x<y$, then $p\Vdash x\in\dot S$.
\item $\Vdash \dot S$ is downwards closed. 
\item For all $p\in\D_T$ and all $x\in\dom(p)$, $p\Vdash I_{\dot S}(x)\subset p(x)$. 
\item For all $p\in\D_T$ and all $x\in\dom(p)$, $p\Vdash p(x)\subset\dot S$.
\item For all $\alpha<\w_1$, the set $\{p\in\D_T:\exists x\in\dom(p)(\h(x)\ge\alpha)\}$ is dense in $\D_T$.
\end{enumerate}
\end{lem}

\begin{proof}Write $\D=\D_T$ to simplify the notation.
\begin{enumerate}[(i),font=\normalfont]
\item It suffices to show that the set $\{r\in \D: x\in\dom(r)\}$ is dense below $p$. Fix $q\le p$ with $x\not\in\dom(q)$. Let $E=\{t\in\dom(q):x<t\}$. Since $y\in E$, $E\neq\emptyset$. Define a function $r$ with domain $\dom(q)\cup\{x\}$ by $r\restr\dom(q)=q$ and $r(x)=\{t\restr(\h(x)+1):t\in E\}$. Since $E\neq\emptyset$, $r(x)\neq\emptyset$. Obviously, $q\subset r$. We check that $r$ is a condition. Fix $u,v\in \dom(r)$ with $u<v$.
\begin{itemize}
\item If $u,v\in\dom(q)$, then $v\restr (\h(u)+1)\in q(u)=r(u)$.
\item If $u=x$, then $v\in E$, so $v\restr (\h(u)+1)\in r(x)$ by definition of $r(x)$.
\item If $v=x$, then $u<y$ by $x<y$, therefore
\begin{align*}
v\restr(\h(u)+1)&=x\restr(\h(u)+1)\\
&= y\restr(\h(u+1))\in q(u)
=r(u).
\end{align*}
\end{itemize}
So, in every case, $v\restr(\h(u)+1)\in r(u)$, and therefore $r$ is a condition.
\item Let $G$ is ${\mathbb{P}}$-generic over $V$ and $y\in S=\dot{S}_G$. Suppose $x<y$. Fix $p\in G$ with $y\in\dom(p)$. By (i), we can find $q\in G$ with $x\in\dom(p)$. Then $x\in S$.
\item Let $G\ni p$ be ${\mathbb{P}}$-generic over $V$ and suppose $y\in I_{S}(x)$, where $S=\dot{S}_G$. Since $y\in S$, there exists $q\in G$ with $y\in\dom(q)$. Choose $r\in G$ with $r\le p,q$. Then $x,y\in \dom(r)$ and $x<y$, so
\[
y=y\restr(\h(x)+1)\in r(x)=p(x)
\]
because $r$ is a condition.
\item Fix $y\in p(x)$. As in (i), it suffices to argue that $\{r\in\D:y\in\dom(r)\}$ is dense below $p$. Let $q\le p$ with $y\not\in\dom(q)$ and let $E=\{t\in\dom(q):y<t\}$.

\underline{Case 1:} $E\neq\emptyset$. This is as in (i): define $r$ by $r\restr\dom(q)=q$ and $r(y)=\{t\restr (\h(y)+1):t\in E\}$. The verification that $r$ is a condition is exactly the same as in the proof of (i).

\underline{Case 2:} $E=\emptyset$. Let $a$ be any finite non-empty subset of $I_T(y)$ and let $r=q\cup\{(y,a)\}$. It is enough to show that $r$ is a condition. Fix $u,v\in\dom(r)$ with $u<v$.

\begin{itemize}
\item If $u,v\in\dom(q)$, then $v\restr(\h(u)+1)\in q(u)=r(u)$.
\item If $u=y$, then $v\in\dom(q)$, so $v\in E$, which contradicts $E=\emptyset$.
\item If $v=y$, then $u\le x$ because $y\in I_T(x)$. We now distinguish two cases. \underline{If $u=x$}, then
\begin{align*}
v\restr (\h(u)+1)&= y\restr(\h(u)+1)\\
&=y\in p(x)= q(x)=r(x)=r(u).
\end{align*}
\underline{If $u<x$}, then $u\in\dom(q)$ and $x\in \dom(p)\subset\dom(q)$, so
\begin{align*}
v\restr(\h(u)+1) &= y\restr (\h(u)+1) = x\restr (\h(u)+1)\in q(u)=r(u).
\end{align*}
\end{itemize}
In every case, $v\restr (\h(u)+1)\in r(u)$, and so $r$ is a condition.
\item Fix $p\in\D$. We may assume that $\dom(p)\subset T\restr\alpha$. Since $\emptyset\not\in\ran(p)$, we can define $\gamma=\max\{\h(y):y\in\bigcup \ran(p)\}$. Note that $\gamma=\beta+1$ for some $\beta$. Fix $y\in\bigcup\ran(p)$ with $\h(y)=\beta+1$. By maximality, $y\not\in\dom(p)$.

\underline{Case 1:} $\beta+1=\alpha$. By the proof of (v), there exists $q\le p$ with $y\in\dom(p)$, and we are done.

\underline{Case 2:} $\beta+1<\alpha$ Since $T$ is normal, there exists $z\in T_\alpha$ with $y<z$. Let $a$ be any finite non-empty subset of $I(z)$ and let $q=p\cup\{(z,a)\}$. We check that $q$ is a condition.
\begin{itemize}
\item If $u,v\in\dom(p)$, then it is easy.
\item If $u=z$, then $v\in\dom(p)$. But $v>u$, so $\h(v)>\alpha>\gamma$, contradiction.
\item If $v=z$, then $\h(u)\le \h(x)$, where $x$ is the immediate predecessor of $y$. Since $T$ is a tree, $u\le x$. If $u=x$, then by $z>y$ we infer that
\[
v\restr(\h(x)+1)=y\in p(x)=q(x).
\]
If $u<x$, then 
\[
v\restr (\h(u)+1)=x\restr(\h(u)+1)=p(u)=r(u)
\]
\end{itemize}
In every case, $v\restr (\h(u)+1)\in r(u)$, and so $r$ is a condition.\qedhere
\end{enumerate}
\end{proof}

\begin{thm}\label{addfbs}
Let $T$ be an infinitely splitting Aronszajn tree and $\dot S$ a $\D_T$ name for the set $\bigcup_{p\in\dot G}\dom(p)$, where $\dot G$ is a $\D_T$-name for the generic filter. Then $\Vdash_{\D_T}``\dot S$ is an uncountable finitely splitting subtree of $T$". 
\end{thm}

\begin{proof}
By Lemma \ref{ccc}, $\D_T$ preserves cardinals. By parts (ii),(iii) and (v) of Lemma \ref{dandelion}.
\end{proof}

Since the forcing $\D_T$ is only ccc and not obviously Knaster, it is unclear whether $T$ remains Aronszajn in the $\D_T$-extension. To deal with this issue, we first do some preliminary forcing.

\begin{defi}[Baumgartner]\label{sp}
If $T$ is a tree, then ${\mathbb{S}}_T$ is the poset of finite order preserving partial functions from $T$ into $\Q$ (the set of rational numbers), ordered by inclusion.
\end{defi}

\begin{thm}[Baumgartner]\label{spccc}
Let $T$ be an Aronszajn tree. Then
\begin{enumerate}
\item ${\mathbb{S}}_T$ has the ccc.
\item $\Vdash_{{\mathbb{S}}_T} T$ is special.
\end{enumerate}
\end{thm}

For a proof, see \cite[Lemma III.5.19]{kunenSetTheory2011}.

\begin{thm}\label{destroyL}
Let $T$ be an Aronszajn tree. Then $\D_T$ adds no cofinal branches to $T$. Therefore, $\D_T$ forces that $T$ is a non-Lindel\"of Aronszajn tree.
\end{thm}

\begin{proof}
Let $G$ be $\D_T$-generic over $V$. Suppose towards a contradiction that $T$ is not an Aronszajn tree in $V[G]$. Note that the definitions of $\mathbb{S}_T$ and $\D_T$ are both absolute, so the two posets are the same whether computed in $V$ or in any outer model. Let $H$ be $\mathbb{S}_T$-generic over $V[G]$. By the Product Lemma, $V[G][H]=V[H][G]$ and $G$ is $\D_T$-generic over $V[H]$. In $V$, $T$ has no uncountable branches, so $\mathbb{S}_T$ is ccc, and therefore $T$ is special in $V[H]$. In particular, $T$ is an Aronszajn tree in $V[H]$, hence $\mathbb{D}_T$ is ccc in $V[H]$. But if $T$ is special then it remains special in any $\aleph_1$-preserving extension, in particular $T$ is special in $V[H][G]$. This contradicts the assumption that there is a cofinal branch through $T$ in $V[G]$.

The second assertion in the theorem statement is an immediate consequence of the first and Theorem \ref{addfbs}.
\end{proof}

%
%

\begin{cor}\label{MA}
Suppose that ${\sf MA}_{\aleph_1}$ holds. Then there are no Lindelöf trees.
\end{cor}

\begin{proof}
Given an Aronszajn tree $T$, we only need to meet $\aleph_1$-many dense sets of $\D_T$ to obtain an uncountable finitely splitting subtree of $T$, namely those in Lemma \ref{dandelion}(vi).
\end{proof}


%

%


\section*{Open questions}

\begin{enumerate}
\item Suppose $\lambda<\mu$ are infinite cardinals. Is it consistent, modulo large cardinals, that there exists a cardinal $\kappa\le \lambda$ which is $\lambda$-square compact but not $\mu$-square compact?

\item For which pairs of infinite cardinals $\kappa,\lambda$ with $\kappa<\lambda$ can one find a Hausdorff $\kappa$-compact space of weight $\lambda$?

\item Does $\sf ZFC$ prove the existence of a non-Lindelöf special tree?

\item Is the topological square of a Lindelöf tree Lindelöf?
\end{enumerate}

\printbibliography

\end{document}